\documentclass[sn-mathphys]{sn}% Math and Physical Sciences Reference Style 
\usepackage{latexsym,amssymb,mathrsfs,amsmath,bbm,amsbsy,mathdots,color,amscd,amsfonts,geometry,graphicx}
\DeclareMathOperator{\Supp}{\operatorname {supp}}
\DeclareMathOperator{\Deg}{\operatorname {deg}}

\def\Delta{\varDelta}
\def\Gamma{\varGamma}
\def\Sigma{\varSigma}
\def\Phi{\varPhi}
\theoremstyle{thmstyleone}%
\newtheorem{theorem}{Theorem}%  meant for continuous numbers
\newtheorem{corollary}[theorem]{Corollary}%
\theoremstyle{thmstyletwo}%
\newtheorem{example}{Example}%
\newtheorem{remark}{Remark}%
\theoremstyle{thmstylethree}%
\newtheorem{definition}{Definition}%

\begin{document}

\title[Some exact constants  for bilateral approximations]{Some exact constants for bilateral approximations in quasi-normed groups}
\author{\fnm{Oleh} \sur{Lopushansky}}
\footnotetext{\href{mailto:olopuszanski@ur.edu.pl}; Institute of Mathematics, University of Rzesz\'ow, Poland}
\abstract{
We establish inverse and direct theorems on best approximations in quasi-normed Abelian groups 
in the form of bilateral Bernstein-Jackson inequalities with exact constants.  
Using integral representations for quasi-norms of  functions $f$ in Lebesgue's spaces
by decreasing rearrangements $f^*$ with the help of approximation $E$-functionals, error estimates are found. 
Examples of numerical calculations,  spectral approximations
of self-adjoint operators obtained by obtained estimates are given.\\}

\keywords{Bernstein-Jackson inequalities, quasi-normed Abelian groups, 
    exact approximation  constants}
\pacs[MSC Classification]{42B35, 41A44, 41A17,46B70}

\maketitle
%\tableofcontents

\section{Introduction and main results}\label{1}

This article is devoted to the proof of a bilateral version of Bernstein-Jackson inequalities for quasi-normed
Abelian groups. The case of similar inequalities on Gaussian Hilbert spaces of random variables was analyzed in the previous publication \cite{Lopushansky2023}.
In the studied case, we use the approximation scale
\[
 \mathcal{B}_\tau^s(\mathfrak{A}_0,\mathfrak{A}_1)=\left(\mathfrak{A}_0,\mathfrak{A}_1\right)_{\vartheta,q}^{1/\vartheta},\quad s+1=1/\vartheta,\quad \tau=\vartheta q
\]
of compatible couple of quasi-normed Abelian groups $(\mathfrak{A}_\imath,{\vert\cdot\vert_\imath})$
with ${0<\vartheta<1}$ and ${0<q\le\infty}$, defined by the Lions-Peetre method of real interpolation.
It should be noted that in partial cases this scale coincides with the known scale of Besov spaces
(see  \cite[Thm 7.2.4]{bergh76}, \cite{DeVore1988,DeVore1998},\cite[Thm 2.5.4]{Nikolski75}).

We also analyze the problem of accurate error estimates in quasi-normed spaces
$L^p_\mu=L^p_\mu(\mathfrak{A})$ $(0\le p\le\infty)$  of   $\mathfrak{A}$-valued functions ${f\colon X\to\mathfrak{A}}$ on a measure space $(X,\mu)$.
Using the decreasing rearrangement method (see e.g. \cite{Burchard,Grafakos2014}),
we reduce this problem to the simpler case for uniquely defined decreasing functions $f^ * $ on $ (0,\infty)$.

The scale of quasi-normed spaces $\mathcal{B}_\tau^s(\mathfrak{A}_0,\mathfrak{A}_1)$ endowed with the quasinorm ${\|\cdot\|_{\mathcal{B}_\tau^s}}$
is determined by the best approximation $E$-functional
\[
E(t,a;\mathfrak{A}_0,\mathfrak{A}_1)= \inf\left\{\vert a-a_0\vert_1\colon\vert a_0\vert _0<t\right\},\quad a\in{\mathfrak{A}_1}
\]
which fully characterizes the accuracy of error estimates (see e.g. \cite{bergh76,Maligranda1991,PeetreSparr1972}).
In Theorem~\ref{t2.1} we establish the bilateral version of Bernstein-Jackson inequalities for
quasi-normed groups with the exact constants $c_{s,\tau}$
depend only on basis parameters $\mathcal{B}_\tau^s(\mathfrak{A}_0,\mathfrak{A}_1)$
\begin{align*}
t^{s}E(t,a)\le{c_{s,\tau}}\vert a\vert_{\mathcal{B}_{\tau}^{s}}\le2^{1/2}\vert a\vert^s_0\vert a\vert_1,\quad
c_{s,\tau}=\bigg[\frac{s}{\tau(s+1)^2}\bigg]^{1/\tau}
\end{align*}
for all $a\in\mathfrak{A}_0\bigcap\mathfrak{A}_1$ and $t,\tau>0$.
The left inequality allows the unique extension to the whole approximation spaces $\mathcal{B}_\tau^s(\mathfrak{A}_0,\mathfrak{A}_1)$
 in form of the Jackson-type inequality
\[
E(t,a)\le{c_{s,\tau}t^{-s}}\,\vert a\vert_{\mathcal{B}_{\tau}^{s}}\quad\text{for all}\quad{a\in \mathcal{B}_{\tau}^{s}(\mathfrak{A}_0,\mathfrak{A}_1)}.
\]

In Section~\ref{sec3} these results are applied to the quasi-normed Abelian groups
$L^p_\mu=L^p_\mu(\mathfrak{A})$ with $0\le p\le\infty$ and a positive Radon measure
$\mu$ on a measure space $X$ of  measurable functions ${f\colon X\to\mathfrak{A}}$ endowed with the $\kappa$-norm
\begin{equation*}%\label{Lp}
\|f\|_p= \left\{\begin{array}{ll}\displaystyle
\Big(\int\vert f(x)\vert ^p\,\mu(dx)\Big)^{1/p} & \hbox{if }  0< p<\infty\\
\mathop{\rm ess\, sup}_{x\in X}\vert f(x)\vert & \hbox{if }p=\infty\\[1ex]
\mu\left(\Supp f\right) & \hbox{if }p=0,
\end{array}\right.\end{equation*}
where  $\Supp f\subset X$ is a measurable subset such that $f\mid_{X\setminus\Supp f}=0$
and ${f\ne 0}$ almost everywhere on $\Supp f$ with respect to the measure $\mu$.
In Theorem~\ref{1} the equality
\begin{equation*}\label{E0infty}
f^*(t)=E\big(t,f;L_\mu^0,L_\mu^\infty\big),\quad f\in{L_\mu^\infty},
\end{equation*}
where $f^*$ is the  decreasing rearrangement  of $\mathfrak{A}$-valued functions $f\in L^p_\mu$
and, as a consequence, the following equalities
 \[
 \|f\|_{ \tau/s}=\left\{
\begin{array}{ll}\displaystyle
\left(\int_0^\infty \left[t^sf^*(t)\right]^\tau\frac{dt}{t}\right)^{1/\tau}&\hbox{if }\tau<\infty\\[2ex]
\sup\limits_{0<t<\infty}t^sf^*(t)&\hbox{if }\tau=\infty,
\end{array}\right.
 \]
are established.  We also compute exact constants in the Jackson-type inequalities
  \begin{equation*}\label{jackson}
f^*(t)\le  t^{s}2^{(s+1)/2}c_{s,\tau}\|f\|_{ \tau/s}\quad\text{for all}\quad{f\in L^{ \tau/s}_\mu}
\end{equation*}
that estimate measurable errors for decreasing rearrangements $f^*$,
  as well as the bilateral Bernstein-Jackson inequalities
\begin{align*}\label{bernstein}
t^{-s}2^{-(s+1)/2}f^*(t)\le c_{s,\tau}\|f\|_{ \tau/s}\le\|f\|^{s}_0\|f\|_\infty\quad\text{for all}\quad{f\in L^0_\mu\cap L^\infty_\mu}.
\end{align*}
In particular, the following inequality holds,
\[
f^*(t)\le \frac{2}{\pi t}\,\int\vert f\vert\,\mu(dx)\quad\text{for all}\quad f\in L^1_\mu.
\]

The numerical algorithm resulting from the Theorem~\ref{1}
was carried out on the example of inverse Gaussian distribution.
Examples of applications for accurate estimates of spectral approximations of self-adjoint operators are given
in Section~\ref{sec4}.

Note that basic  notations used in this work can be found in \cite{bergh76,Triebel78}.

\section{Best approximation scales of quasi-normed Abelian groups}\label{2}

In what follows, we study the
$\kappa$-normed Abelian groups $(\mathfrak{A},{\vert\cdot\vert})$ relative to an operation  $"+"$,
where $\kappa$-norm ${\vert\cdot\vert}$ with the constant
$\kappa\ge1$ is determined by the assumptions (see e.g. \cite{bergh76,Maligranda1991}):
\begin{quote}%\begin{enumerate}
$\vert0\vert=0$ and $\vert a\vert>0$ for all nonzero $a\in \mathfrak{A}$, \\
$\vert a\vert=\vert-a\vert$ for all  $a\in \mathfrak{A}$,\\
$\vert a+b\vert\le\kappa(\vert a\vert+\vert b\vert)$ with  $\kappa$ independent on $a,b\in \mathfrak{A}$
%\end{enumerate}
\end{quote}
As is known, each $\kappa$-norm $\vert\cdot\vert$ can be replaced by an
equivalent $1$-norm $\vert\cdot\vert'$ such that
$\vert\cdot\vert'\le\vert\cdot\vert^\rho\le2\vert\cdot\vert'$ by taking
$(2\kappa)^\rho=2$ with a suitable $\rho>0$  (see  e.g.  \cite{PeetreSparr1972}).
We will not assume the completeness of groups.

Given a compatible couple $(\mathfrak{A}_0, \mathfrak{A}_1)$ of $\kappa_\imath$-normed groups
$(\mathfrak{A}_\imath,{\vert\cdot\vert_\imath})$ with $\imath=\{0,1\}$ and $a=a_0+a_1$ in the
sum ${\mathfrak{A}_0+\mathfrak{A}_1}$ such that $a_\imath\in \mathfrak{A}_\imath$ we define the best approximation $E$-functional
$E(t,a;\mathfrak{A}_0,\mathfrak{A}_1)$ with ${a\in \mathfrak{A}_0+\mathfrak{A}_1}$ and ${t>0}$ in the form (see \cite[no.7]{bergh76})
\begin{equation}\label{E2}
E(t,a)=E(t,a;\mathfrak{A}_0,\mathfrak{A}_1)= \inf\left\{\vert a-a_0\vert_1\colon\vert a_0\vert _0<t\right\},\quad a\in{\mathfrak{A}_1}.
\end{equation}

\begin{definition}\label{Bes0}
  For any $\tau=q\vartheta$ and $s+1=1/\vartheta$ with ${0<\vartheta<1}$ and  ${0<q\le\infty}$
  the $\kappa$-normed best approximation scales of Abelian groups are defined to be
 \begin{equation}\label{bes}
 \begin{split}
\mathcal{B}_{\tau}^s(\mathfrak{A}_0,\mathfrak{A}_1)&=\left\{a\in \mathfrak{A}_0+\mathfrak{A}_1\colon\vert a\vert_{\mathcal{B}_\tau^s}<\infty\right\},\\
\vert a\vert_{\mathcal{B}_\tau^s}&=\left\{
\begin{array}{ll}
\displaystyle\left(\int_0^\infty \left[t^s E(t,a)\right]^\tau
\frac{dt}{t}\right)^{1/\tau}&\!\!\!\!\!\!\hbox{if }\tau<\infty\\[2ex]
\sup\limits_{0<t<\infty}t^s E(t,a)&\!\!\!\!\!\!\hbox{if }\tau=\infty,
\end{array}\right.
\end{split}
\end{equation}
where (by  \cite[Lemma 7.1.6]{bergh76})
$\kappa=2\kappa_1\max(\kappa_0^s,\kappa_1^s)\max(1,2^{-1/q'})\max(1,2^{s-1})$, $(1/q'=1-1/q).$
\end{definition}

Given a compatible couple $(\mathfrak{A}_0, \mathfrak{A}_1)$,  we also define
the quadratic functional of Lions-Peetre's type (see e.g.  \cite{McLean2000,PeetreSparr1972})
\[
K_2(t,a)=K_2(t,a;\mathfrak{A}_0, \mathfrak{A}_1)=
\inf\limits_{a=a_0+a_1}\left(\vert a_0\vert_0^2+t^2\vert a_1\vert_1^2\right)^{1/2},
\quad t>0.
\]
We will use the quadratically modified real interpolation method.
Define the interpolation Abelian group of  Lions-Peetre's type $\left(\mathfrak{A}_0,\mathfrak{A}_1\right)_{\vartheta,q}$ endowed
with the quasinorm ${\|\cdot\|_{(\mathfrak{A}_0,\mathfrak{A}_1)_{\vartheta,q}}}$,
\begin{align*}
K_{\theta,q}\left(\mathfrak{A}_0,\mathfrak{A}_1\right)&=\left(\mathfrak{A}_0,\mathfrak{A}_1\right)_{\vartheta,q}
=\big\{a\in \mathfrak{A}_0+\mathfrak{A}_1 \big\}, \\
\vert a\vert_{(\mathfrak{A}_0,\mathfrak{A}_1)_{\vartheta,q}} &=\left\{
\begin{array}{ll}
\displaystyle{\bigg(\int_0^\infty \left[t^{-\vartheta}
K_2(t,a)\right]^q\frac{dt}{t}\bigg)^{1/q}}&\!\!\!\!\hbox{if }q<\infty\\[2ex]
\sup\limits_{0<t<\infty}t^{-\vartheta} K_2(t,a)&\!\!\!\!\hbox{if }
q=\infty.
\end{array}\right.
\end{align*}

Following \cite{DL19,Lopushansky2023}, we extend the use of  approximation constants $c_{s,\tau}$ to the case of arbitrary quasi-normed Abelian groups $\mathcal{B}_{\tau}^s(\mathfrak{A}_0,\mathfrak{A}_1)$ described in the classic works \cite{PeetreSparr1972,pietsch1981}, where
\begin{equation}\label{crucial}
 c_{s,\tau} :=\left\{\begin{array}{cl}\displaystyle
\left[\frac{s}{\tau(s+1)^2}\right]^{1/\tau}&\hbox{if } \tau<\infty \\\displaystyle
   1 &\hbox{if }  \tau=\infty
 \end{array}\right.
\end{equation}
is determined  by the normalization factor $N_{\vartheta,q}$ from \cite[Thm 3.4.1]{bergh76} to be
\begin{equation}\label{crucial1}
c_{s,\tau}=N^{1/q}_{\vartheta,q}\left(\vartheta q^2\right)^{-1/q\vartheta},\quad N_{\vartheta,q}:=[q\vartheta(1-\vartheta)]^{1/q}.
\end{equation}
The following theorem establishes the bilateral form of approximation inequalities with exact constants
for the case of quasi-normed Abelian groups.

\begin{theorem}\label{t2.1}
The bilateral Bernstein-Jackson inequalities  with the exact constant ${c}_{s,\tau}$
\begin{align}\label{bernstein}
t^{s}E(t,a)&\le{c}_{s,\tau}\vert a\vert_{\mathcal{B}_{\tau}^{s}}\le2^{1/2}\vert a\vert^s_0\vert a\vert_1
\quad\text{for all}\quad a\in \mathfrak{A}_0\cap \mathfrak{A}_1
\end{align}
 hold. There is a unique extension of Jackson's inequality on the whole approximation scale
\begin{align}\label{jackson}
E(t,a)&\le t^{-s}{c}_{s,\tau}
\vert a\vert_{\mathcal{B}_{\tau}^{s}}\quad\text{for
all}\quad {a\in \mathcal{B}_{\tau}^{s}(\mathfrak{A}_0, \mathfrak{A}_1)}, \ \ r>0.
\end{align}
 \end{theorem}

\begin{proof}
Let ${0< q <\infty}$ and $\alpha=\vert a\vert_{\mathfrak{A}_1}/\vert a\vert_{\mathfrak{A}_0}$ with a nonzero $a\in \mathfrak{A}_0\cap \mathfrak{A}_1$. Since
\begin{align*}
K_2(t,a)^2&=
\inf\limits_{a=a_0+a_1}\left(\vert a_0\vert_0^2+t^2\vert a_1\vert_1^2\right)\le\vert a\vert^2_0\min(1,\alpha^2t^2)={\min\left(\vert a\vert^2_0,t^2\vert a\vert^2_1\right)}
\end{align*}
or otherwise $K_2(t,a)\le{\vert a\vert_0\min(1,\alpha t)}
={\min\left(\vert a\vert_0,t\vert a\vert_1\right)},$ we get the inequality
\[
\vert a\vert^q_{(\mathfrak{A}_0,\mathfrak{A}_1)_{\vartheta,q}}\le\vert a\vert_1^q\int_0^\alpha
t^{-1+q(1-\vartheta)}dt+\vert a\vert^q_0\int_\alpha^\infty t^{-1-\vartheta q}dt.
\]
Calculating the integrals and using that $\alpha=\vert a\vert_1/\vert a\vert_0$, we obtain
\begin{align*}
\vert a\vert^q_{(\mathfrak{A}_0,\mathfrak{A}_1)_{\vartheta,q}}&\le\frac{\alpha^{q(1-\vartheta)}}{q(1-\vartheta)}
\vert a\vert_1^q+\frac{\alpha^{-\vartheta q}}{\vartheta q}\vert a\vert^q_0=\frac{1}{q\vartheta(1-\vartheta)}\left(\vert a\vert^{1-\vartheta}_0\vert a\vert_1^\vartheta\right)^q.
\end{align*}

Let it be now $q =\infty$. Then  the inequality
\[
t^{-\vartheta}K_2(t,a)\le{\min\left(t^{-\vartheta}\vert a\vert_0,t^{1-\vartheta}\vert a\vert_1\right)}
\]
holds. Taking in this case $t=\vert a\vert_0/\vert a\vert_1$, we obtain that $t^{-\vartheta}K_2(t,a)\le\vert a\vert^{1-\vartheta}_0\vert a\vert_1^\vartheta.$
Combining previous inequalities and taking into account \eqref{crucial1}, we have
\begin{equation}\label{inq0001}
\begin{split}
\|a\|_{(\mathfrak{A}_0,\mathfrak{A}_1)_{\vartheta,q}}&\le \left\{\begin{array}{ll}
  N_{\vartheta,q}^{-1}\vert a\vert^{1-\vartheta}_0\vert a\vert_1^\vartheta&\hbox{if } q <\infty \\[1.5ex]
\vert a\vert^{1-\vartheta}_0\vert a\vert_1^\vartheta&\hbox{if }q=\infty
\end{array}\right.,\quad a\in \mathfrak{A}_0\cap \mathfrak{A}_1.
\end{split}\end{equation}

For further considerations, we need the functional
\[
K_\infty(v,a):=\inf\limits_{a=a_0+a_1}\max\big(\vert a_0\vert_{\mathfrak{A}_0},v\vert a_1\vert_{\mathfrak{A}_1}\big).
\]
Now we will use the known properties of the considered functionals that
\begin{align*}
v^{-\theta}K_\infty(v,f)\to 0\quad \text{as \ } {v\to 0} \text{ or }  {v\to\infty}\\
{t^{-1+1/\theta}E(t,f)\to 0}\quad \text{as \ } {v\to 0} \text{ or }  {v\to\infty}
\end{align*}
(see \cite[Thm 7.1.7]{bergh76}). As a result, we get
\begin{align*}
\int_0^\infty\left(v^{-\vartheta}K_\infty(v,a)\right)^q\frac{dv}{v}&=-\frac{1}{\vartheta q}\int_0^\infty K_\infty(v,a)^qdv^{-\vartheta q}=\frac{1}{\vartheta q}\int_0^\infty v^{-\vartheta
 q}dK_\infty(v,a)^q.
\end{align*}
On the other hand,
integrating by parts with the change  $v=t/E(t,a)$, we get for $s+1=1/\vartheta$ that
\begin{align*}\label{18.5}
\int_0^\infty\left(v^{-\vartheta}K_\infty(v,a)\right)^q\frac{dv}{v}
&=\frac{1}{\vartheta q}\int_0^\infty\left(t/E(t,a)\right)^{-\vartheta q}dt^q=\frac{1}{\vartheta
q^2}\int_0^\infty\left(t^sE(t,a)\right)^{\vartheta q}\frac{dt}{t}.
\end{align*}

From the  definition  $K_\infty$ and $K_2$, it follows that
\begin{equation}\label{inqK1}
K_\infty(t,a)\leq K_2(t,a)\leq 2^{1/2}K_\infty(t,a)
\end{equation}
(see \cite[Remark 3.1]{PeetreSparr1972}). Now, taking into account that
\[
\frac{1}{\vartheta q^2}\|a\|^{\vartheta q}_{\mathcal{B}_{\tau}^{s}}=\frac{1}{\vartheta
q^2}\int_0^\infty\left(t^sE(t,a)\right)^{\vartheta q}\frac{dt}{t}
=\int_0^\infty\left(v^{-\vartheta}K_\infty(v,a)\right)^q\frac{dv}{v},
\]
according to the left inequality from \eqref{inqK1}, we obtain
\begin{equation}\label{20.5}
\frac{1}{\vartheta q^2}\vert a\vert^{\vartheta q}_{\mathcal
B_{\tau}^{s}}
\le\int_0^\infty\left(v^{-\vartheta}K_2(v,a)\right)^q\frac{dv}{v}=\vert a\vert^q_{(\mathfrak{A}_0,\mathfrak{A}_1)_{\vartheta,q}}.
\end{equation}
By using the right inequality from \eqref{inqK1}, we have
\begin{align*}
\vert a\vert^q_{(\mathfrak{A}_0,\mathfrak{A}_1)_{\vartheta,q}}
=\int_0^\infty\left(v^{-\vartheta}K_2(v,a)\right)^q\frac{dv}{v}
\le 2^{q/2}\int_0^\infty\left(v^{-\vartheta}K_\infty(v,a)\right)^q\frac{dv}{v} \\
=\frac{2^{q/2}}{\vartheta q^2}\int_0^\infty\left(t^sE(t,a)\right)^{\vartheta
q}\frac{dt}{t}
=\frac{2^{q/2}}{\vartheta q^2}\vert a\vert^{\vartheta
q}_{\mathcal B_{\tau}^{s}}.
\end{align*}
As a result, by virtue of \eqref{20.5} we have the inequalities
\begin{equation}\label{inqeqvintBes1}
\vert a\vert^q_{(\mathfrak{A}_0,\mathfrak{A}_1)_{\vartheta,q}}
\le\frac{2^{q/2}}{\vartheta q^2}\vert a\vert^{\vartheta q}_{\mathcal B_{\tau}^{s}}\le
2^{q/2}\vert a\vert^q_{(\mathfrak{A}_0,\mathfrak{A}_1)_{\vartheta,q}}\quad\text{with}\quad\tau=\vartheta q.
\end{equation}
That gives the following isomorphism with equivalent quasinorms
\begin{equation}\label{main}
\mathcal{B}_\tau^s(\mathfrak{A}_0,\mathfrak{A}_1)=\left(\mathfrak{A}_0,\mathfrak{A}_1\right)_{\vartheta,q}^{1/\vartheta}.
\end{equation}

By \cite[Lemma 7.1.2 and Thm 7.1.1]{bergh76} for each ${v>0}$ there is ${t>0}$ such that
\begin{equation}\label{inqEK1}
(t^sE(t,a))^{\vartheta}\le
v^{-\vartheta}K_\infty(v,a)\le\left(t^sE(t-0,a)\right)^\vartheta.
\end{equation}
Since $\vert a\vert^{\vartheta}_{\mathcal{B}_\infty^{s}}\le\vert a\vert_{(\mathfrak{A}_0,\mathfrak{A}_1)_{\vartheta,q}}$
for $q =\infty$, the inequalities \eqref{inqEK1} for any ${v>0}$ yield
\begin{align*}
v^{-\vartheta} K_2(v,a)&\le v^{-\vartheta} 2^{1/2}K_\infty(v,a)\le 2^{1/2}\big(t^sE(t-0,a)\big)^\vartheta\\
&\le2^{1/2}\Big(\sup_{t>0}\,t^{s}E(t,a)\Big)^\vartheta=2^{1/2}\vert a\vert^\vartheta_{\mathcal{B}_\infty^{s}}.
\end{align*}
As a result,
$\vert a\vert_{(\mathfrak{A}_0,\mathfrak{A}_1)_{\vartheta,q}}
\le2^{1/2}\vert a\vert^{\vartheta}_{\mathcal{B}_\infty^s}$  and the isomorphism \eqref{main} holds for $q=\infty$.

Prove the Bernstein inequality \eqref{bernstein}.  Combining  \eqref{inq0001} and \eqref{20.5}, we obtain
\begin{equation}\label{inq0002}
\vert a\vert^{\vartheta}_{\mathcal B_{\tau}^s}\le
\left\{\begin{array}{ll}\displaystyle
\left[\frac{q}{1-\vartheta}\right]^{1/q}\vert a\vert^{1-\vartheta}_0\vert a\vert_1^\vartheta&\hbox{if } q <\infty \\[2ex]
\vert a\vert^{1-\vartheta}_0\vert a\vert_1^\vartheta&\hbox{if }
q=\infty.
\end{array}\right.
\end{equation}
Thus, the second inequality \eqref{bernstein} we get by setting $s+1=1/\vartheta$ and  $\tau=\vartheta q$.

Prove  the inequality \eqref{jackson}. Integrating  ${\min(1,v/t)K_2(t,a)\le{K}_2(v,a)}$, we obtain
\[
\int_0^\infty\left(v^{-\vartheta}\min(1,v/t)\right)^q\frac{dv}{v}K_2(t,a)^q
\le\int_0^\infty\left(v^{-\vartheta}K_2(v,a)\right)^q\frac{dv}{v}=\vert a\vert_{(\mathfrak{A}_0,\mathfrak{A}_1)_{\vartheta,q}}^q.
\]
Since the left integral can be rewritten as the following sum
\[\begin{split}
\int_0^\infty\left(v^{-\vartheta}\min(1,v/t)\right)^q\frac{dv}{v}
&=\int_0^tv^{(1-\vartheta)q-1}t^{-q}dv+\int_t^\infty v^{-\vartheta q-1}dv=\frac{1}{t^\vartheta N_{\vartheta,q}^{1/q}},
\end{split}\]
we obtain the inequality
\[\begin{split}
\bigg(\int_0^t\frac{ v^{(1-\vartheta)q-1}}{t^q}dv+\int_t^\infty
v^{-\vartheta
q-1}dv\bigg)^{1/q}K_2(t,a)&=\frac{K_2(t,a)}{t^\vartheta
N_{\vartheta,q}}\le\vert x\vert_{(\mathfrak{A}_0,\mathfrak{A}_1)_{\vartheta,q}}.
\end{split}\]
As a result, $K_2(t,a)\le t^\vartheta N_{\vartheta,q}
\vert a\vert_{(\mathfrak{A}_0,\mathfrak{A}_1)_{\vartheta,q}}$. Hence, taking into account
(\ref{inqK1}) and  (\ref{inqEK1}), we get
\begin{equation*}
v^{1-\vartheta}E(v,a)^{\vartheta}\le
t^{-\vartheta}K_\infty(t,a)\le N_{\vartheta,q}
\vert a\vert_{(\mathfrak{A}_0,\mathfrak{A}_1)_{\vartheta,q}}.
\end{equation*}
Applying \eqref{inqeqvintBes1}, we obtain
\[
v^{1-\vartheta}E(v,a)^{\vartheta}\le{2^{1/2}}N_{\vartheta,q}
\vert a\vert^\vartheta_{\mathcal B_{\tau}^s}.
\]
Taking into account  that $\vartheta=1/(s+1)=q/\tau,$
we obtain the inequality \eqref{jackson} for all ${0<q<\infty}$.

In the case $q =\infty$, we have
\[
t^{s}E(t,a)\le \sup_{t>0}\,t^{s}E(t,a)=\vert a\vert_{\mathcal{B}_\infty^s}\]
for all ${a\in\mathcal{B}_\infty^s}$.  Thus, the  inequality
\eqref{jackson} holds for both cases.

Finally note that Bernstein-Jackson inequalities are achieved at $\tau = \infty$, so they are sharp.
\end{proof}

\begin{corollary}\label{c2.0}
If  the quasi-normed Abelian groups $\mathfrak{A}_0$, $\mathfrak{A}_1$ are compatible then the following isomorphism  with equivalent quasi-norms holds,
\begin{equation}\label{main1}
 \mathcal{B}_\tau^s(\mathfrak{A}_0,\mathfrak{A}_1)=\left(\mathfrak{A}_0,\mathfrak{A}_1\right)_{\vartheta,q}^{1/\vartheta},\quad s+1=1/\vartheta,\quad \tau=\vartheta q.
\end{equation}
If compatible groups $\mathfrak{A}_0$ and $\mathfrak{A}_1$ are complete, the approximation group $\mathcal{B}_\tau^s(\mathfrak{A}_0,\mathfrak{A}_1)$ is complete.
\end{corollary}
\begin{proof}
The isomorphism \eqref{main1} immediately follows from \eqref{main} and Theorem~\ref{t2.1}. Note that if both compatible
groups $\mathfrak{A}_0$ and $\mathfrak{A}_1$ are complete, the interpolation group
$\left(\mathfrak{A}_0,\mathfrak{A}_1\right)_{\vartheta,q}$ is complete \cite[Thm 3.4.2 \& Lemma 3.10.2]{bergh76}.
By Theorem~\ref{t2.1} $\mathcal{B}_\tau^s(\mathfrak{A}_0,\mathfrak{A}_1)$ is complete.
\end{proof}

\begin{corollary}\label{2.4}
Let $(\mathfrak{A}_0,\mathfrak{A}_1)$, $(\mathfrak{A}'_0,\mathfrak{A}'_1)$ be compatible.
If $T\colon \mathfrak{A}_\imath\to\mathfrak{A}'_\imath$ $(\imath=0,1)$
are quasi-Abelian  mappings, i.e. ${T(a_0+a_1)}={b_0+b_1}$  and
$\vert Ta_\imath\vert_{\mathfrak{A}_\imath}\le \kappa_\imath\vert b_\imath\vert_{\mathfrak{A}'_\imath}$
with constants $\kappa_\imath$
 \cite[p.81]{bergh76}, then there exists a constant $C>0$ such that
\[T\colon \mathcal{B}_\tau^s(\mathfrak{A}_0,\mathfrak{A}_1)\to
\mathcal{B}_\tau^s(\mathfrak{A}'_0,\mathfrak{A}'_1)\quad \text{and}\quad\vert Ta\vert_{\mathcal{B}_\tau^s(\mathfrak{A}'_0,\mathfrak{A}'_1)}\le
C\vert a\vert_{\mathcal{B}_\tau^s(\mathfrak{A}_0,\mathfrak{A}_1)}.\]
\end{corollary}
\begin{proof}
It follows from \eqref{inqeqvintBes1}, \eqref{main}  and the boundness  of
\[
T\colon(\mathfrak{A}_0,\mathfrak{A}_1)_{\vartheta,q}\to(\mathfrak{A}'_0,\mathfrak{A}'_1)_{\vartheta,q},
\]
where $C\le4^{1/\vartheta}(\vartheta q^2)^{-1/\vartheta q}\kappa_0^s\,\kappa_1$ that gives the required inequality.
\end{proof}

\begin{corollary}\label{c2.2}
Let $0\le\vartheta_0<\vartheta_1\le1$, ${0<q\le\infty}$ and $\vartheta=(1-\eta)\vartheta_0+\eta\vartheta_1$
${(0<\eta<1)}$. If the quasi-normed Abelian groups
$(\mathfrak{A}_0, \mathfrak{A}_1)$, $(\mathfrak{A}'_0, \mathfrak{A}'_1)$ are compatible and
 the inclusions \[(\mathfrak{A}_0, \mathfrak{A}_1)_{\vartheta_\imath,q}\subset\mathfrak{A}'_\imath\subset(\mathfrak{A}_0, \mathfrak{A}_1)_{\vartheta_\imath,\infty},\quad(\imath=0,1)\]
are valid, the following isomorphism  with equivalent quasinorms holds,
\begin{equation}\label{reit}
 \mathcal{B}_{\vartheta q}^{(1-\vartheta)/\vartheta}(\mathfrak{A}_0,\mathfrak{A}_1)=\mathcal{B}_{\eta q}^{(1-\vartheta)/\vartheta}(\mathfrak{A}'_0,\mathfrak{A}'_1).
\end{equation}
\end{corollary}

\begin{proof}
Using  \eqref{main} and the known equality
$\left(\mathfrak{A}_0,\mathfrak{A}_1\right)_{\vartheta,q}=\left(\mathfrak{A}'_0,\mathfrak{A}'_1\right)_{\eta,q}$
\cite{PeetreSparr1972}, \cite[Theorem 3]{Komatsu1981} in the case ${r=1}$, we  get the
reiteration identity \eqref{reit}  with accuracy to quasinorm
equivalency.
\end{proof}

\begin{corollary}\label{c2.3}
The following continuous embedding hold
\[ \mathcal{B}_{\vartheta q}^{(1-\vartheta)/\vartheta}(\mathfrak{A}_0,\mathfrak{A}_1)\looparrowright\mathcal{B}_{\vartheta p}^{(1-\vartheta)/\vartheta}(\mathfrak{A}_0,\mathfrak{A}_1)\quad
\text{with}\quad q<p.\]
\end{corollary}
\begin{proof}
It  follows from \eqref{main}, since  for arbitrary $q<p$ the
following continuous embedding
$(\mathfrak{A}_0,\mathfrak{A}_1)_{\vartheta,q}\looparrowright(\mathfrak{A}_0,\mathfrak{A}_1)_{\vartheta,p}$ is
true by virtue of \cite[Lemma p. 385]{Komatsu1981}.
\end{proof}

\begin{corollary}
If $\vartheta_0<\vartheta_1$ then the embedding $\mathfrak{A}_1\subset \mathfrak{A}_0$
yields
\[\mathcal{B}_{\vartheta_1 q}^{(1-\vartheta_1)/\vartheta_1}(\mathfrak{A}_0,\mathfrak{A}_1)\subset\mathcal{B}_{\vartheta_0 q}^{(1-\vartheta_0)/\vartheta_0}(\mathfrak{A}_0,\mathfrak{A}_1).\]
\end{corollary}
\begin{proof}
It follows from Corollary~\ref{c2.2} and \cite[Theorem 3.4.1(d)]{bergh76}.
\end{proof}

\section{Exact constants in bilateral error estimates}\label{sec3}

Let ${(\mathfrak{A},\vert\cdot\vert)}$ be a $\kappa$-normed Abelian group.
Consider a quasi-normed space $L^p_\mu=L^p_\mu(\mathfrak{A})$ with a positive Radon measure
$\mu$ on a measure space $(X,\mu)$ of  measurable functions ${f\colon X\to\mathfrak{A}}$ endowed with the $\kappa$-norm
\begin{equation*}%\label{Lp}
\|f\|_p= \left\{\begin{array}{ll}\displaystyle
\Big(\int\vert f(x)\vert ^p\,\mu(dx)\Big)^{1/p} & \hbox{if }  0< p<\infty\\
\mathop{\rm ess\, sup}_{x\in X}\vert f(x)\vert & \hbox{if }p=\infty\\[1ex]
\mu\left(\Supp f\right) & \hbox{if }p=0,
\end{array}\right.\end{equation*}
where  $\Supp f\subset X$ is a measurable subset such that $f\mid_{X\setminus\Supp f}=0$
and ${f\ne 0}$ almost everywhere on $\Supp f$ with respect to the measure $\mu$.

For a $\mu$-measurable function ${f\colon X\to\mathfrak{A}}$, we define the distribution function
\[
m(\sigma,f)=\mu\left\{x\in X\colon {\vert f(x)\vert>\sigma}\right\}.
\]
Denote by $f^*$ the decreasing rearrangement of $f$, where
\[
f^*(t):=\inf\left\{\sigma>0\colon m(\sigma,f_\sigma)\le t\right\}
\]
with $f_\sigma(x)=f(x)$ if $\vert f(x)\vert>\sigma$ and $f_\sigma(x)=0$ otherwise
(for details see \cite[no~\!1.3]{bergh76}).
The decreasing rearrangement $f^*$ is nonnegative nonincreasing continuous on the right function of $\sigma$ on $(0,\infty)$
which is equimeasurable with $f$ in the sense that
\[
m(\sigma,f)=m(\sigma,f^*)\quad  \text{for all}\quad \sigma\ge 0.
\]

In other words, the decreasing rearrangement of a function $f$ is a generalized inverse of its
distribution function  in the sense that if $m(\sigma,f)$ is one-to-one then $f^*$ is simply the inverse of $m(\sigma,f)$.

Since $f$ is rearrangeable,
$\left\{\sigma>0\colon m(\sigma,f_\sigma)\le t\right\}$ is nonempty for $t>0$ thus $f^*$ is finite on its domain.

In the case where $\mu(X)<\infty$, we consider $f^*$ as a function on $(0,\mu(X))$, since
$f^*(t)=0$ for all $t>\mu(X)$.
The set $\left\{\sigma>0\colon m(\sigma,f_\sigma)=0\right\}$ can be empty. If $f$ is bounded then
$f^*\to\|f\|_\infty$ as $t\to0_+$,
otherwise, $f^*$ is unbounded at the origin.

As is also know (see e.g. \cite{ONeilWeiss1963}) for $f\in L^p_\mu$ with $p\ge1$ that
\[
\int\vert f\vert^p\,\mu(dx)=p\int_{0}^{\infty}\sigma^{p-1}m(\sigma,f)\,d\sigma,\quad \|f\|_p^p=\int_{0}^{\infty}f^*(t)^p\,dt.
\]
The couple
$\big(L^0_\mu,L^\infty_\mu\big)$  is compatible in the interpolation theory sense (see e.g. \cite[Lemma 3.10.3]{bergh76} or  \cite[no~\!1]{Komatsu1981}).
Given  elements $f={f_0+f_\infty}$ in the algebraic  sum ${L^0_\mu+L^\infty_\mu}$,  we define the
  best approximation $E$-functional $E\left(t,f;L^0_\mu,L^\infty_\mu\right)$ with ${t>0}$ as (see e.g. \cite[no~\!7]{bergh76})
\begin{align*}%\label{E2}
E(t,f)&=E\big(t,f;L^0_\mu,L^\infty_\mu\big)= \inf\left\{\|f-f_0\|_\infty\colon\|f_0\|_0<t\right\},\quad f\in{L^\infty_\mu}.
\end{align*}

In what follows, we use the  equivalent parameters
 ${0<\vartheta<1}, \ {0<q\le\infty}$, where $s+1=1/\vartheta$ and $\tau=\vartheta q$.
 Following \cite{Lopushansky2023}, the appropriate system of exact constants $C_{\theta,q}$ is defined by the normalization factor
\begin{equation*}%%\label{crucial2}
N_{\theta,q}=\left(\int_{0}^{\infty}\vert t^{-\theta}\mathcal{N}(t)\vert^q\frac{dt}{t}\right)^{-1/q}\quad
\text{with}\quad \mathcal{N}(t)=\frac{t}{\sqrt{1+t^2}}
 \end{equation*}
in the known interpolation Lions-Peetre method  to be
\begin{equation}\label{crucial}
 \begin{split}
 C_{\theta,q}&=\left\{
 \begin{array}{cl}\displaystyle
 \frac{2^{1/2\theta}}{(q^2\theta)^{1/q\theta}}N_{\theta,q}^{1/\theta}&\hbox{if }q<\infty,\\[2ex]
 \displaystyle\left(\frac{\sin\pi\theta}{\pi\theta}\right)^{1/2\theta}&\hbox{if }q=2,\\[2ex]
 \displaystyle
 2^{1/2\theta}&\hbox{if }q=\infty.
 \end{array}\right.\end{split}
 \end{equation}

\begin{remark}
For this equivalent parameter system
the sharp constant $c_{s,\tau}$ takes an equivalent form
\begin{equation*}
2^{-1/2\theta}C_{\theta,q}=\left\{\begin{array}{cl}
c_{s,\tau} &\hbox{if } \tau<\infty \\
   1 &\hbox{if }  \tau=\infty
 \end{array}\right.,
\end{equation*}
where the following dependencies between indexes are performed,
\begin{equation*}
2^{-1/2\theta}C_{\theta,q}=\left[(1-\theta)/q\right]^{1/q\theta}=(\theta q^2)^{-1/q\theta} N_{\theta,q}^{1/\vartheta}.
 \end{equation*}
\end{remark}
A graphical image of the normalized sharp constant $ c_{s,\tau}$  in variables $s,\tau$ shown on Fig.~\!1.
\begin{figure}
\centering
\begin{minipage}{.5\textwidth}
 \includegraphics[height=2.5cm,width=.9\linewidth]{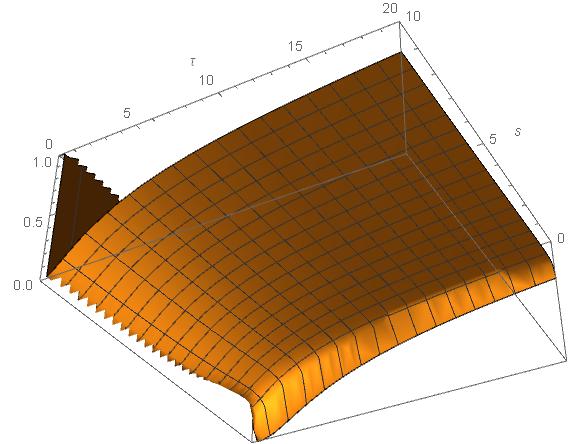}
  \end{minipage}%
\begin{minipage}{.5\textwidth}
  \includegraphics[width=.9\linewidth]{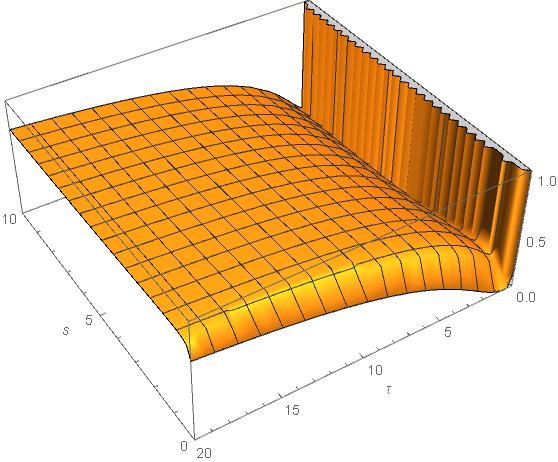}
\end{minipage}
\centering
\caption{3D-graph of $c_{s,\tau}=(s/\tau)^{1/\tau}(s+1)^{-2/\tau}$ for ${s,\tau>0}$}
\end{figure}

\begin{theorem}\label{1} {\rm(a)}
For any  ${0<\vartheta< 1}$ and   ${0<q\le\infty}$ the equality
\begin{equation}\label{E0infty}
f^*(t)=E\big(t,f;L^0_\mu,L^\infty_\mu\big),\quad f\in{L^\infty_\mu}
\end{equation}
and, as a consequence, the following equalities hold,
 \begin{equation}\label{besov}
 \|f\|_{q\theta^2/(1-\theta)}=\left\{
\begin{array}{ll}\displaystyle
\left(\int_0^\infty \left[t^{-1+1/\theta}f^*(t)\right]^{ q\theta}\frac{dt}{t}\right)^{1/q\theta}&\hbox{if }q<\infty\\[2ex]
\sup\limits_{0<t<\infty}t^{-1+1/\theta}f^*(t)&\hbox{if }q=\infty.
\end{array}\right.
 \end{equation}

{\rm(b)}
The  Jackson-type  inequality
\begin{equation}\label{jackson1}
f^*(t)\le  t^{1-1/\theta}C_{\theta,q}\|f\|_{ q\theta^2/(1-\theta)}\quad\text{for all}\quad{f\in L_\mu^{q\theta^2/(1-\theta)}},
\end{equation}
as well as, the following  bilateral Bernstein-Jackson-type inequalities are valid,
\begin{align}\label{bernstein1}
t^{-1+1/\theta}f^*(t)&\le {C}_{\theta,q}\|f\|_{q\vartheta^2/(1-\theta)}\le2^{1/2\theta}\|f\|^{-1+1/\theta}_0\|f\|_\infty
\text{ for all } f\in L^0_\mu.
\end{align}
\end{theorem}

\begin{proof} {\rm(a)}
By definition, the functional $E(t,f)$ is the infimum of $\|f-f_0\|_\infty$ such that ${\mu(\Supp f_0)\le t}$.
The next reasoning extends \cite[Lemma 7.2.1]{bergh76} on the case of $\mathfrak{A}$-valued functions.

Put $g_0(x)=f(x)$ on the subset $\Supp f$ and let $g_0(x)=0$ outside $\Supp f$.
Then we obtain ${\|f-g_0\|_\infty}\le{\|f-g\|_\infty}$.

In a similar way, let
$g_\sigma(x)=f(x)$ if $\vert f(x)\vert >\sigma$ and ${g_\sigma(x)=0}$ otherwise.
Then  for the number
$\tau={\sup\big\{\vert f(x)\vert\colon x\ne \Supp f\big\}}$ we get
$\Supp f_\tau\subset \Supp f$. It follows ${\mu(\Supp g_\tau)\le t}$. Since $\|f-g_\tau\|_\infty\le\tau$ and $\|f-g_0\|_\infty=\tau$,
we obtain
\[
E\big(t,f;L^0_\mu,L^\infty_\mu\big)={\inf}_\sigma\left\{\|f-g_\sigma\|_\infty\colon \mu\left(\Supp g_\sigma\right)\le t\right\},
\]
where the rigth hand side is equal to $f^*$. As a result, the equality \eqref{E0infty} holds.

Using the  best approximation $E$-functional, we define  the  quasi-normed space
 \begin{align}\label{bes}
E_{\theta,q}\big(L^0_\mu,L^\infty_\mu\big)&=\left\{f\in {L^0_\mu+L^\infty_\mu}\colon\|f\|_{E_{\theta,q}}<\infty\right\},
\\
\|f\|_{E_{\theta,q}}&=\left\{
\begin{array}{ll}\displaystyle
\left(\int_0^\infty \left[t^{-1+1/\theta}E(t,f)\right]^{ q\theta}\frac{dt}{t}\right)^{1/q\theta}&\hbox{if }q<\infty,\\[2ex]
\sup\limits_{0<t<\infty}t^{-1+1/\theta} E(t,f)&\hbox{if }q=\infty.
\end{array}\right.
\end{align}
By the known approximation theorem (see \cite[Thm 7.2.2]{bergh76}) the isometric isomorphism
\begin{equation}\label{isometry}
E_{\theta,q}\big(L^0_\mu,L^\infty_\mu\big)=L_\mu^{ q\theta^2/(1-\theta)},\quad \|f\|_{E_{\theta,q}}=\|f\|_{\vartheta^2 q/(1-\theta)}
\end{equation}
 holds for all ${f\in L_\mu^\infty}.$
Combining equalities \eqref{E0infty}, \eqref{bes} and \eqref{isometry}, we  get  \eqref{besov}.

{\rm(b)} In what follows, we use the classic integral of the real interpolation method
\begin{equation}\label{real}
\|f\|_{\theta,q}=\left(\int_{0}^{\infty}\left\vert t^{-\theta} f(x)\right\vert^q\frac{dt}{t}\right)^{1/q},
\quad {0<\theta<1}, \ 0< q<\infty.
\end{equation}
Let us consider the quadratic  $K_2$-functional (see e.g.  \cite[App. B]{McLean2000}) for
the interpolation couple of quasi-normed groups $\big(L^0,L^\infty\big)$,
\begin{align*}
K_2(t,f)&={K}_2\big(t,f;L^0_\mu,L^\infty_\mu\big)\\
&=\inf_{f=f_0+f_\infty}\left\{\left(\|f_0\|_0^2+t^2\|f_\infty\|_\infty^2\right)^{1/2}\colon
f_0\in L^0_\mu, \ f_\infty\in L^\infty_\mu\right\},
\end{align*}
determining the real interpolation quasi-normed Abelian  group
\begin{align*}
\big(L^0_\mu,L^\infty_\mu\big)_{\theta,q}:=
K_{\theta,q}\big(L^0_\mu,L^\infty_\mu\big)=\left\{f=f_0+f_\infty\colon\|K_2(\cdot,f)\|_{\theta,q}<\infty\right\}
\end{align*}
which is endowed with the norm
\begin{align*}
\|f\|_{K_{\theta,q}}&=\left\{
\begin{array}{ll}
N_{\theta,q}\|K_2(\cdot,f)\|_{\theta,q}&\hbox{if }q<\infty\\[1.5ex]
\sup\limits_{t\in(0,\infty)}t^{-\theta} K_2(t,f)&\hbox{if }
q=\infty.
\end{array}\right.\end{align*}

Note that for any  ${0<\vartheta< 1}$ and   ${q=2}$ the following equalities hold,
 \begin{align*}
 N_{\theta,2}=\|\mathcal{N}\|_{\theta,2}^{-1}= \|K_2(\cdot,1)\|_{\theta,2}^{-1}=\left(\frac{2\sin\pi\theta}{\pi}\right)^{1/2}.
 \end{align*}
In fact, the relationship between the weight function $\mathcal{N}(t)^2=t^2/(1+t^2)$ and the quadratic $K_2$-functional
is explained by the formula
\[
N_{\theta,2}=\|\mathcal{N}\|_{\theta,2}^{-1}= \|K_2(\cdot,1)\|_{\theta,2}^{-1}
\]
(see e.g. \cite[Ex. B.4]{McLean2000}). It follows from
\[
\min_{z=z_0+z_1}\left(\alpha_0\vert z_0\vert^2+\alpha_1\vert z_1\vert^2\right)={\alpha_0\alpha_1\vert z\vert^2}{\alpha_0+\alpha_1}
\]
for a fixed $\alpha_0,\alpha_1>0$ and a complex $z$. This minimum is achieved when
\[
\alpha_0z_0=\alpha_1z_1=\frac{\alpha_0\alpha_1z}{(\alpha_0+\alpha_1)}.
\]
Thus, $K_2(t,1)$ is minimized when $f_0$, $f_1$ are such that \[f_0=t^2f_1=\frac{t^2}{1+t^2}.\]

By integrating the above functions (see e.g. \cite[Ex. B.5, Thm B.7]{McLean2000}) it follows that
the normalization factor $N_{\theta,2}$  is equal to
$N_{\theta,2}=\left({2\sin\pi\theta}/{\pi}\right)^{1/2}.$

In particular,  the equalities \eqref{E0infty},
\begin{align*}
f^*(t)&=\inf\big\{\sigma\colon m(\sigma,f_\sigma)\le t\big\}\\[1ex]
&= \inf\left\{\|f-f_0\|_\infty\colon\|f_0\|_0<t\right\}=E\big(t,f;L^0_\mu,L^\infty_\mu\big),\quad f\in{L_\mu^\infty},
\end{align*}
as well as, the isometric isomorphism \eqref{isometry},
\begin{align*}
E_{\theta,q}\big(L^0_\mu,L^\infty_\mu\big)=L_\mu^{ q\theta^2/(1-\theta)},\quad \|f\|_{E_{\theta,q}}
=\|f\|_{\vartheta^2 q/(1-\theta)}
\end{align*}
allows calculating the exact form of  best constants \eqref{crucial}. Now, applying Theorem~\ref{t2.1} to
$\big(L^0_\mu,L^\infty_\mu\big)$, we obtain inequalities \eqref{jackson1}  and \eqref{bernstein1}. 
\end{proof}

\begin{corollary}\label{maincor}
The decreasing rearrangement for $q=2$  has the estimation
\begin{align}\label{q2}
f^*(t)&\le t^{1-1/\theta}\left(\frac{\sin\pi\theta}{\pi\theta}\right)^{1/2\theta}\!\|f\|_{2\theta^2/(1-\theta)}\quad
\text{for all}\quad {f\in L_\mu^{q\theta^2/(1-\theta)}}
\end{align}
which for $\theta=1/2$ can be written as follows
\begin{align}\label{l1}
f^*(t)&\le \frac{2}{\pi t}\,\int\vert f\vert\,\mu(dx)\quad\text{for all}\quad f\in L^1_\mu.
\end{align}
\end{corollary}

Note that if $\theta\to1$ then $\sin\pi\theta\to0$ thus above inequalities disappear.
One has a sense only for the quasi-normed group $L_\mu^{2\theta^2/(1-\theta)}$ with $0<\theta<1$.

\begin{corollary}
The equality \eqref{besov}  with $s+1=1/\vartheta$ and $\tau=\vartheta q$ takes the form
 \[
 \|f\|_{ \tau/s}=\left\{
\begin{array}{ll}\displaystyle
\left(\int_0^\infty \left[t^sf^*(t)\right]^\tau\frac{dt}{t}\right)^{1/\tau}&\hbox{if }\tau<\infty\\[2ex]
\sup\limits_{0<t<\infty}t^sf^*(t)&\hbox{if }\tau=\infty
\end{array}\right.
 \]
 for any $f\in L^{\tau/s}_\mu$. Then the  Jackson-type  inequality takes the form
\begin{equation}\label{jackson2}
f^*(t)\le  t^{-s}2^{(s+1)/2}\,c_{s,\tau}\|f\|_{ \tau/s}\quad\text{for all}\quad{f\in L^{ \tau/s}_\mu}
\end{equation}
 and bilateral Bernstein-Jackson-type inequalities take the form
\begin{align}\label{bernstein2}
t^{s}2^{-(s+1)/2}f^*(t)&\le c_{s,\tau}\|f\|_{ \tau/s}
\le\|f\|^{s}_0\|f\|_\infty\quad\text{for all}\quad{f\in L^0_\mu\cap L^\infty_\mu}.
\end{align}
\end{corollary}
It instantly follows from the relations
\begin{equation*}
C_{\theta,q}=\left\{\begin{array}{cl}
2^{(s+1)/2}\,c_{s,\tau} &\hbox{if } \tau<\infty \\
2^{(s+1)/2} &\hbox{if }  \tau=\infty
 \end{array}\right.,\quad\text{where}\quad \vartheta=\frac{1}{s+1},\quad 2^{1/2\theta}=2^{(s+1)/2}.
\end{equation*}

\begin{remark}\label{r4}
The inequalities \eqref{bernstein} is an extension on the  case measurable functions the known
best approximation Bernstein-Jackson inequalities.
Whereas the inequality \eqref{jackson} is an extension of the approximation Jackson inequality.
The scale of approximation quasi-normed Abelian  groups is often denoted as
\[
\mathcal{B}_\tau^s(L^0_\mu,L^\infty_\mu):=E_{\theta,q}(L^0_\mu,L^\infty_\mu),
\]
where the space $\mathcal{B}_\tau^s$ coincides with a suitable extension of the Besov type quasi-normed groups (see e.g. \cite[Thm 7.2.4]{bergh76}, \cite[Def.~\!4.2.1/1]{Triebel78}).
The scale of  quasi-normed Besov Abelian  groups for another  approximation  couples  is described in \cite{DL19,Feichtinger2016,Pesenson2024}.
\end{remark}

\section{Examples of bilateral estimates with exact constants}\label{sec4}

Taking into account the previous statements, we can give typical examples of Bernstein-Jackson inequalities
with explicit constants for different types of best approximations.

\begin{example}[Applications to classic Besov scales]
Let ${\mathfrak{A}_0= L^p(\mathbb{R}^n)}$ with ${(1<p<\infty)}$.
As is known (see e.g. \cite{Nikolski75})  each real-valued function ${a\in L^p(\mathbb{R}^n)}$
can be approximated by entire analytic functions $g\in\mathcal{E}_p^t(\mathbb{C}^n)$, where
$\mathcal{E}_p^t(\mathbb{C}^n)$ means the space of entire analytic functions on $\mathbb{C}^n$
of an exponential type ${t>0}$ with restrictions to  $\mathbb{R}^n$ belonging to $L^p(\mathbb{R}^n)$.
The best approximations can be characterized by the best approximation functional
\[
E\left(t,a;\mathcal{E}_p,L^p(\mathbb{R}^n)\right)= \inf_{t>0}\left\{\|a-g\|_{L^p(\mathbb{R}^n)}\colon g\in\mathcal{E}_p^t(\mathbb{C}^n)\right\},
\]
where the subspace  ${\mathfrak{A}_1= \mathcal{E}_p}$ with
$\mathcal{E}_p={\bigcup}_{t>0} \mathcal{E}_p^t(\mathbb{C}^n)$ is endowed with the quasinorm
\[
\vert g\vert_{\mathcal{E}_p}=\|g\|_{L^p(\mathbb{R}^n)}+\left\{\sup{\vert\zeta\vert\colon\zeta \in\Supp\hat{g}}\right\},\quad g\in\mathcal{E}_p
\]
defined using the support $\Supp\hat{g}$ of the Fourier-image $\hat{g}$.
Then according to Theorem~\ref{t2.1} the corresponding approximation inequalities take the form
\begin{align*}
t^{s}E\left(t,a;\mathcal{E}_p,L^p(\mathbb{R}^n)\right)&\le
c_{s,\tau}\|f\|_{B_{p,\tau}^s(\mathbb{R}^n)}\le2^{1/2}
\vert a\vert^s_{\mathcal{E}_p}\|a\|_{L^p}, \  {a\in \mathcal{E}_p\cap L^p(\mathbb{R}^n)},\\
E\left(t,a;\mathcal{E}_p,L^p(\mathbb{R}^n)\right)&\le{t^{-s}c_{s,\tau}}\,\|a\|_{B_{p,\tau}^s(\mathbb{R}^n)}, \quad
{a\in {B_{p,\tau}^s(\mathbb{R}^n)}},
\end{align*}
In this case the approximation scale
\[
\mathcal{B}_\tau^s\big(\mathcal{E}_p,L^p(\mathbb{R}^n)\big)
=\left(\mathcal{E}_p, L^p(\mathbb{R}^n)\right)_{\vartheta,q}^{1/\vartheta}
\]
exactly coincides with the classic Besov scale denoted by $B_{p,\tau}^s(\mathbb{R}^n)$  (see \cite[p.197]{Triebel78}).
\end{example}

\begin{example}[Applications to scales of periodic functions]
Following e.g. \cite[no 1.5]{bergh76}, \cite{Prestin} we can write Bernstein-Jackson inequalities in a more general form. Let
$X=\mathbb{T}$ be the $1$-dimensional torus and the Hilbert space $\mathfrak{A}_1=L^2(\mathbb{T})$ has the orthonormal basis
$\left\{\mathfrak{e}_k=e^{2\pi\mathrm{i}kt}\colon k\in\mathbb{Z},\,t\in\mathbb{T}\right\}$.
Let $\mathfrak{A}_0=\{\text{trigonometric polynomials \ } a_0\}$ with the quasi-norm $\vert a_0\vert_0=\Deg a_0$,  and
$\mathfrak{A}_1=\{\text{$2\pi$-periodic functions \ } a\}$ with the norm $\vert a\vert_1=\|a\|_{L^2(\mathbb{T})}$,  as well as,
$E(n,a;\mathfrak{A}_0,\mathfrak{A}_1)= \inf\left\{\|a-a_0\|_{L^2(\mathbb{T})}\colon\Deg a _0<n\right\}$ for all ${a\in L^2(\mathbb{T})}$.
Let $ c_{j,\tau} :=\left\{\begin{array}{cl}\displaystyle
\left[\frac{j}{\tau(j+1)^2}\right]^{1/\tau}&\hbox{if } \tau<\infty \\\displaystyle
   1 &\hbox{if }  \tau=\infty\end{array}\right.$ with $j\in\mathbb{N}$.
Then Bernstein-Jackson inequalities have the form
\begin{align*}
n^{j}E(n,a)&\le{c}_{j,\tau}\vert a\vert_{\mathcal{B}_{\tau}^{j}}\le2^{1/2}\vert a\vert^j_0\vert a\vert_1
\quad\text{for all}\quad a\in \mathfrak{A}_0\cap \mathfrak{A}_1,\\
E(n,a)&\le n^{-j}{c}_{j,\tau}
\vert a\vert_{\mathcal{B}_{\tau}^{j}}\quad\text{for
all}\quad {a\in \mathcal{B}_{\tau}^{j}(\mathfrak{A}_0, \mathfrak{A}_1)},
\end{align*}
where $\mathcal{B}_\tau^s(\mathfrak{A}_0,\mathfrak{A}_1)=\left(\mathfrak{A}_0,\mathfrak{A}_1\right)_{\vartheta,q}^{1/\vartheta}$ with
 $j+1=1/\vartheta\in\mathbb{N}$, $\tau=\vartheta q$.

For $\mathfrak{A}_1=L^\infty(\mathbb{T})$ with $\vert a\vert_1=\|a\|_{L^\infty(\mathbb{T})}$
and $\mathfrak{A}_0=L^0(\mathbb{T})$ with $\vert a_0\vert_0=\Deg a_0$, we consider
 the Banach space $L^\infty(\mathbb{T})$-valued functions ${f\colon X\to L^\infty(\mathbb{T})}$.
Then bilateral Bernstein-Jackson inequalities have the form
\begin{align*}
n^{-j}2^{-(j+1)/2}f^*(n)&\le c_{j,\tau}\|f\|_{ \tau/j}
\le\|f\|^{j}_0\|f\|_\infty,\quad{f\in L^0\cap L^\infty},\\
f^*(n)&\le n^{j}2^{(j+1)/2}c_{j,\tau}\|f\|_{ \tau/j},\quad{f\in L^{ \tau/j}},
\end{align*}
where $f^*(n)=E\big(n,f;L^0,L^\infty\big)=\inf\big\{\sigma\colon m(\sigma,f_\sigma)\le t\big\}$ for all $f\in{L^\infty}$.
\end{example}

\begin{example}[Applications to operator spectral approximations]
Consider the example commonly used in  foundations of quantum systems
for describing quantum states (see e.g. \cite[no~\!3.4]{Hall2013}).
Similar estimates of spectral approximations were early analyzed in the paper \cite{DL19}.
In what follows, we consider the spaces $L^p=L^p(\mathbb{R};\mathbb{R})$
of real-valued functions.

Let $H$ be a Hilbert complex space with the norm $\|\cdot\|_H={\langle\cdot\mid\cdot\rangle^{1/2}}$ and
a self-adjoint unbounded
linear operator $T$ with the dense domain $\mathcal{D}(T)\subset H$ is given.
By spectral theorem the measurable function ${f}$ from $T$  can be well defined using the following spectral expansions
(see e.g. \cite{Hall2013})
\[
T=\int_{\sigma(T)}\lambda\,\mu(d\lambda),\quad {f}(T)=\int_{\sigma(T)}{f}(\lambda)\,\mu(d\lambda),
\]
 where $\mu$ is a unique projection-valued measure determined on its spectrum $\sigma(T)\subset\mathbb{R}$
with values in the Banach space of bounded linear operators $\mathcal{L}(H)$, which can be extended on $\mathbb{R}\setminus\sigma(A)$  as zero.

At beginning, we consider the case when
for any Borel set $\Sigma\subset\mathbb{R}$ and any $\psi\in\mathcal{D}(A)$ the $\mathcal{L}(H)$-valued measurable function
\[
\Sigma\longmapsto f(\Sigma)=\int_{\Sigma}{f}(\lambda)\,\mu_\psi(d\lambda)
\]
belongs to $L^\infty$ with respect to  the positive measure
$\mu_\psi(d\lambda):=\big\langle\mu(d\lambda) \psi\mid\psi\big\rangle$, where $\|\psi\|_H=1$.

Consider the corresponding quadratic forms  $f_\psi(\Sigma)=\big\langle f(\Sigma)\psi\mid\psi\big\rangle$ as measurable functions from $L^2$
and let $f^*_\psi$ be the suitable decreasing rearrangement.
By Theorem~\ref{main} the inequalities \eqref{jackson} and \eqref{bernstein} in this case take the form
\begin{align}\label{jackson3}
f_\psi^*(t)&\le  t^{1-1/\theta}C_{\theta,q}\|f_\psi\|_{ q\theta^2/(1-\theta)},
\quad{f_\psi\in L_\mu^{ q\theta^2/(1-\theta)}},\\[1.5ex]\label{bernstein3}
t^{-1+1/\theta}f_\psi^*(t)&\le {C}_{\theta,q}\|f_\psi\|_{q\vartheta^2/(1-\theta)}\le2^{1/2\theta}\|f_\psi\|^{-1+1/\theta}_0\|f_\psi\|_\infty
\end{align}
for all $f_\psi\in L^0_\mu\cap L^\infty_\mu$,
where
$\|f_\psi\|_\infty=\sup_{\Sigma\subset\mathbb{R}}\vert  f_\psi(\Sigma)\vert$ and  $\|f_\psi\|_0=\mu_\psi\left(\Supp f_\psi\right)$.
Applying the estimation \eqref{l1} from Corollary~\ref{maincor}, we obtain
\begin{align}\label{spectrum}
f_\psi^*(t)&\le \frac{2}{\pi t}\,\int\vert f(\lambda)\vert \,\mu(d\lambda)
\quad\text{for all}\quad f_\psi\in L^1_\mu.
\end{align}

The inequality \eqref{jackson3},  \eqref{spectrum}  give error estimations
(in quadratic forms terms)  of   spectral approximations
$f_\psi(T)={\langle f(T)\psi\mid\psi\rangle}$
by  $f_\psi(\Sigma)$ with $\mu\left(\Sigma\right)\le t$.
The inequalities \eqref{bernstein1} characterise the approximation accuracy.

In another case, using the $\mathcal{L}(H)$-valued measurable uniformly bounded functions
\[
\tilde{f}\colon\Sigma\longmapsto \int_\Sigma f(\lambda)\,\mu(d\lambda),\quad f\in L^1_\mu
\]
belonging to $L^1_\mu(\mathbb{R},\mathcal{L}(H))$, we same as above obtain the  inequality
\begin{align*}
\tilde{f}^*(t)&\le \frac{2}{\pi t}\Big\|\int f(\lambda)\,\mu(d\lambda)\Big\|_{\mathcal{L}(H)}.
\end{align*}
\end{example}

\begin{example}[Samples of numerical calculations]\label{4.4}
The  inequality \eqref{jackson} together with the exact value of the normalized constant ${c}_{s,\tau}$
are suitable for direct numerical calculations. Illustrate this on several graphs,
taking as an example the inverse Gaussian distribution function.

Consider in $L_\mu^\infty$ with the Gaussian measure $\mu(d\tau)=(2\pi)^{-1/2}e^{-\tau^2/2}$ $(\tau\in\mathbb{R})$
 two-parameter inverse Gaussian distribution with support on $(0,\infty)$ that is given by
\[
f(t)=C\sqrt{\frac{l}{2\pi t^3}}\exp\left(-\frac{l(t-m)^2}{2m^2t}\right),\quad C,t>0
\]
with the mean $m>0$ and the shape parameter $l>0$. By Theorem~\ref{main}
\begin{align*}
f^*(u)=E(u,f)&=\inf\left\{\|f-f_\sigma\|_\infty\colon m(\sigma,f_\sigma)\le u\right\},\\
m(\sigma,f_\sigma)&=\mu\left\{u\in(0,\infty)\colon {\vert f_\sigma(u)\vert>\sigma}\right\}
\end{align*}
coincides with the non-increasing rearrangement of $f$ on $(0,\infty)$. As a result, we obtain
\begin{align*}
f^*(u)=E(u,f)&\le{c}_{s,\tau}
u^{-s}\,\|f\|_{\tau/s}\quad\text{for
all}\quad {f\in L_\mu^{\tau/s}},\end{align*}
where $L_\mu^{\tau/s}=\mathcal{B}_\tau^s\big(L^0_\mu,L^\infty_\mu\big)$
in accordance with Remark~\ref{r4} and the equality \eqref{besov}.

  The obtained results of numerical calculations are illustrated on Fig.~\!2,3.

\begin{figure}
\begin{center}
\includegraphics[height=4cm,width=12cm]{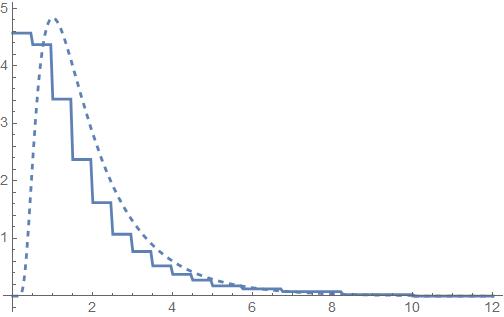}
\caption{Graphs of the function $f(t)=10\sqrt{\frac{2}{\pi t^3}}e^{-\frac{(t-2)^2}{2t}}$ (dots) and its non-increasing rearrangement $f^*(t)$
(polygon),  (${0<t<10}$)}
\end{center}
\end{figure}

%\begin{figure}
%\begin{center}
%\includegraphics[height=2.5cm,width=12cm]{tab1}
%\caption{Table of discrete values $E(u,f)$ and $u^{-s}c_{s,\tau}\|f\|_{\mathcal{B}_{\tau}^{s}}$, (${0<u<11}$, ${s=\tau=2}$)}
%\end{center}
%\end{figure}

\begin{figure}
\begin{center}
\includegraphics[height=4.5cm,width=12cm]{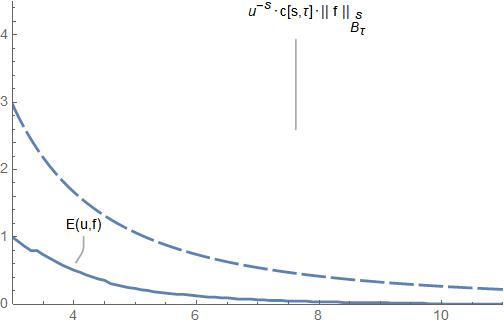}
\caption{Graphs of approximation $E(u,f)$ via $u^{-s}c_{s,\tau}\|f\|_{\mathcal{B}_{\tau}^{s}}$
(${0<u<11}$, ${s=\tau=2}$)  using Jackson's inequality}
\end{center}
\end{figure}
\end{example}
%\newpage

%\section*{Acknowledgments}
%This is acknowledgment text.\cite{Kenamond2013} Provide text here. This is acknowledgment text. Provide text here. This is acknowledgment text. Provide text here. This is acknowledgment text. %Provide text here. This is acknowledgment text. Provide text here. This is acknowledgment text. Provide text here. This is acknowledgment text. Provide text here. This is acknowledgment text. Provide %text here. This is acknowledgment text. Provide text here.

%\subsection*{Author contributions}

%This is an author contribution text. This is an author contribution text. This is an author contribution text. This is an author contribution text. This is an author contribution text.

%\subsection*{Financial disclosure}None reported.

\section*{Statements and Declarations: \normalsize \rm
There are no conflicts and potential competing of interest to disclose.}

\end{document}